\documentclass[preprint,draft]{elsarticle}
\usepackage{amsmath,amsthm}
\usepackage[english]{babel}
\usepackage{amsfonts}
\usepackage{latexsym}

\newtheorem{theorem}{Theorem}

\newcommand*{\cd}{(\cdot)}
\newcommand*{\lt}{L_2(\mathbb R)}
\newcommand*{\Wr}{W_2^r(\mathbb R)}

\newcommand*{\iR}{\int_{\mathbb R}}

\newcommand*{\ov}{\overline}

\newcommand*{\wx}{\widehat x}

\DeclareMathOperator*{\vraisup}{vraisup}

\DeclareMathOperator*{\RE}{Re}

\begin{document}
\begin{frontmatter}
\title{Optimal recovery of linear operators from information of random functions}
\author{K.~Yu.~Osipenko}
\address{Moscow State University,\\
Institute for Information Transmission Problems,
Russian Academy of Sciences, Moscow}
\ead{kosipenko@yahoo.com}
\begin{abstract}
The paper concerns problems of the recovery of linear operators defined on sets of functions from information of these functions given with stochastic errors. The constructed optimal recovery methods, in general, do not use all the available information. As a consequence, optimal methods are obtained for recovering derivatives of functions from Sobolev classes by the information of their Fourier transforms given with stochastic errors. A similar problem is considered for solutions of the heat equation.
\end{abstract}

\begin{keyword}
optimal recovery \sep random functions \sep Fourier transform
\MSC[2010] 41A65 \sep 41A46 \sep 49N30 \sep 60G35
\end{keyword}

\end{frontmatter}

\section{Introduction}

There are several approaches to recovery problems from inaccurate information. One of them concerns the case when the error in the initial data is deterministic. Quite a lot of works are devoted to this case. The main results can be found in \cite{MT2}, \cite{MR}, \cite{Os1}, \cite{Pl}, \cite{TW} and the literature cited there.

Another approach is related to the fact that the initial information is considered to be given with a random error. There are also many papers dedicated to this topic. The following are the closest to the setting under consideration: \cite{Do}, \cite{Do1}, \cite{Wi}, \cite{Wi1}, \cite{Re}, \cite{Kr}. A distinctive specificity of this paper is that the information used here is not random vectors, but random functions.

\section{General setting}

Denote by $\mathcal W$ the set of functions $x\cd\in\lt$ for which
$$\iR\nu(t)|x(t)|^2\,dt<\infty,$$
where $\nu\cd$ is continuous and positive almost everywhere. Put
$$W=\biggl\{\,x\cd\in\mathcal W:\iR\nu(t)|x(t)|^2\,dt\le1\,\biggr\}.$$
Consider the problem of optimal recovery of the operator $\Lambda x\cd=\mu\cd x\cd$ on the class $W$ by functions $x\cd$ given with random errors (we assume that $\mu\cd$ is continuous and such that $\Lambda$ maps $\mathcal W$ into $\lt$). More precisely, for a fixed $\delta>0$ and every $x\cd\in W$ we consider the set of random functions
$$Y_\delta(x\cd)=\{\,y_\xi\cd\in\lt:\mathbb My_\xi\cd=x\cd,\ \mathbb Dy_\xi\cd\le\delta\,\}.$$
We will also assume that the set of these random functions and the corresponding probability measures are such that they allow for a change in integration so that the equalities are valid
\begin{equation}\label{eq1}
\mathbb M\iR p(t)y_\xi(t)\,dt=\iR p(t)\mathbb My_\xi(t)\,dt,\quad p(t)\in\lt,
\end{equation}
and
\begin{equation}\label{eq2}
\mathbb M\iR|y_\xi(t)|^2\,dt=\iR\mathbb M|y_\xi(t)|^2\,dt.
\end{equation}

As recovery methods we consider all possible mappings $\varphi\colon\lt\to\lt$.
The error of a~method~$m$ is defined as
$$e(\Lambda,W,\delta,\varphi)
=\left(\sup_{\substack{x\cd\in W\\y_\xi\cd\in Y_\delta(x\cd)}}\mathbb M\left(\|\Lambda x\cd-\varphi(y_\xi\cd)\cd\|^2_{\lt}\right)\right)^{1/2}.$$
The problem is to find the error of optimal recovery
\begin{equation}\label{bas}
E(\Lambda,W,\delta)=\inf_{\varphi\colon\lt\to\lt}e(\Lambda,W,\delta,\varphi)
\end{equation}
and a method on which this infimum is attained which is called optimal.

We assume that $|\mu\cd|$ and $\nu\cd$ are even functions, $|\mu(t)|>0$ almost everywhere, and $|\mu\cd|/\sqrt{\nu\cd}$ is a monotonically decreasing function on $\mathbb R_+=[0,+\infty)$. Put
\begin{equation}\label{fs}
f(s)=\int_{|t|\le s}\left(\frac{\sqrt{\nu(s)}}{|\mu(s)|}\frac{|\mu(t)|}{\sqrt{\nu(t)}}-
1\right)\nu(t)\,dt.
\end{equation}
It is easy to check that $f\cd$ is a monotonically increasing function. Assume that $f(s)\to+\infty$ as $s\to+\infty$. Then for any $\delta>0$ the equation $f(s)=\delta^{-2}$ has a unique solution $t_\delta$.

\begin{theorem}\label{T1}
For all $\delta>0$
\begin{equation}\label{TE}
E(\Lambda,W,\delta)=\delta\biggl(\int_{|t|\le t_\delta}|\mu(t)|^2\biggl(1-\dfrac{\sqrt{\nu(t)}}{|\mu(t)|}\frac{|\mu(t_\delta)|}
{\sqrt{\nu(t_\delta)}}\biggr)\,dt\biggr)^{1/2}.
\end{equation}
Moreover, the method
$$\varphi(y_\xi\cd)(t)=\left(1-\dfrac{\sqrt{\nu(t)}}{|\mu(t)|}\frac{|\mu(t_\delta)|}{\sqrt{\nu(t_\delta)}}
\right)_+\mu(t)y_\xi(t)$$
is optimal $(a_+=\max\{a,0\})$.
\end{theorem}

\begin{proof}
1. The lower bound. Consider the segment $[-A,A]\subset\mathbb R$. Let us divide it on $2N$ parts by the points
$x_j=\pm j\dfrac AN$, $j=0,1,\ldots,N$. Set $a=(a_1,\ldots,a_{2N})$,
$$x_a(t)=\begin{cases}a_{2j-1},&t\in[x_{j-1},x_j),\ j=1,\ldots N,\\
a_{2N-1},&t=A,\\
a_{2j},&t\in[x_{-j},x_{-j+1}),\ j=1,\ldots,N,\\
0,&|t|>A.\end{cases}$$

Let $\tau=(\tau_1,\ldots,\tau_{2N})$, $\tau_1\ge\tau_2\ge\ldots\ge\tau_{2N}>0$, and $x_\tau\cd\in W$. Put
$$B=\{\,x_a\cd\in W:a_j=\pm\tau_j,\quad j=1,\ldots,2N\,\}.$$
It is obvious that $B\subset W$. Set
$$p_j=\frac{\delta^2}{\delta^2+\tau_j^2},\quad j=1,\ldots,2n.$$
Due to the monotony conditions of $\tau_j$, we have
$$0<p_1\le\ldots\le p_{2n}<1.$$
Any $x\cd\in B$ can be written in the form
$$x\cd=\sum_{j=1}^{2n}s_j(x)\tau_je_j\cd,$$
where $s_j(x)\in\{-1,1\}$, and
$$e_{2j-1}\cd=\chi_{[x_{j-1},x_j)}\cd,\quad e_{2j}\cd=\chi_{[x_{-j},x_{-j+1})}\cd,\quad j=1,\ldots,N$$
($\chi_\Omega\cd$ is the characteristic function of the set $\Omega$).

For each $x\cd\in B$ we define the distribution $\eta(x)\cd$ in the following way:
$$\eta(x)\cd=\begin{cases}0,&\mbox{with probability }p_1,\\
\dfrac{s_1(x)\tau_1}{1-p_1}e_1\cd&\mbox{with probability }p_2-p_1,\\
\displaystyle\sum_{j=1}^2\dfrac{s_j(x)\tau_j}{1-p_j}e_j\cd,&\mbox{with probability }p_3-p_2,\\
\dotfill&\dotfill\\
\displaystyle\sum_{j=1}^{2n-1}\dfrac{s_j(x)\tau_j}{1-p_j}e_j\cd,&\mbox{with probability }p_{2n}-p_{2n-1},\\
\displaystyle\sum_{j=1}^{2n}\dfrac{s_j(x)\tau_j}{1-p_j}e_j\cd,&\mbox{with probability }1-p_{2n}.
\end{cases}$$
Thus, we have the following distribution for $\eta(x)(t)$ if $t\in[x_{j-1},x_j)$:
\begin{multline*}
\eta(x)(t)=\begin{cases}
0,&\mbox{with probability }p_{2j-1},\\
\dfrac{s_{2j-1}(x)\tau_{2j-1}}{1-p_{2j-1}},&\mbox{with probability }1-p_{2j-1},
\end{cases}\\j=1,\ldots,N,
\end{multline*}
and if $t\in[x_{-j},x_{-j+1})$:
$$\eta(x)(t)=\begin{cases}
0,&\mbox{with probability }p_{2j},\\
\dfrac{s_{2j}(x)\tau_{2j}}{1-p_{2j}},&\mbox{with probability }1-p_{2j},
\end{cases}\quad j=1,\ldots,N.$$
It is easy to verify that $\mathbb M(\eta(x)\cd)\cd=x\cd$. Moreover,
$$\mathbb D(\eta(x)\cd)(t)=\begin{cases}\delta^2,&t\in[A,A],\\
0,&|t|>A.\end{cases}$$
Consequently, $\eta(x)\cd\in Y_\delta(x\cd)$ for all $x\cd\in B$.

Let $\varphi$ be an arbitrary recovery method. Taking into account that the set $B$ is finite (with $2^{2N}$ elements), we have
\begin{multline}\label{ver}
e^2(\Lambda,W,\delta,\varphi)\ge\sup_{x\cd\in B}\mathbb M\|\Lambda x\cd-\varphi(\eta(x)\cd)\cd\|_{\lt}^2\\
=\sup_{x\cd\in B}\biggl(\sum_{j=1}^{2N+1}(p_j-p_{j-1})\biggl\|\Lambda x\cd-\varphi\biggl(\sum_{k=1}^{j-1}
\dfrac{s_k(x)\tau_k}{1-p_k}e_k\cd\biggr)\cd\biggr\|_{\lt}^2\biggr)\\
\ge\frac1{2^{2N}}\sum_{x\cd\in B}\biggl(\sum_{j=1}^{2N+1}(p_j-p_{j-1})\biggl\|\Lambda x\cd-\varphi\biggl(\sum_{k=1}^{j-1}
\dfrac{s_k(x)\tau_k}{1-p_k}e_k\cd\biggr)\cd\biggr\|_{\lt}^2\biggr)\\
=\frac1{2^{2N}}\sum_{j=1}^{2N+1}(p_j-p_{j-1})\sum_{x\cd\in B}\biggl\|\Lambda x\cd-\varphi\biggl(\sum_{k=1}^{j-1}
\dfrac{s_k(x)\tau_k}{1-p_k}e_k\cd\biggr)\cd\biggr\|_{\lt}^2;
\end{multline}
here $p_0=0$ and $p_{2N+1}=1$. Set
$$B_{s_1,\ldots,s_{j-1}}=\{\,x\cd\in B:s_1(x)=s_1,\ldots,s_{j-1}(x)=s_{j-1}\,\},$$
$j=1,\ldots,2N+1$ (for $j=1$ this set coincides with $B$). Then
\begin{multline*}
\frac{p_j-p_{j-1}}{2^{2N}}\sum_{x\cd\in B}\biggl\|\Lambda x\cd-\varphi\biggl(\sum_{k=1}^{j-1}
\dfrac{s_k(x)\tau_k}{1-p_k}e_k\cd\biggr)\cd\biggr\|_{\lt}^2
=\frac{p_j-p_{j-1}}{2^{2N}}\\\times\sum_{s_1,\ldots,s_{j-1}}\,\,\sum_{x\in B_{s_1,\ldots,s_{j-1}}}\biggl\|\Lambda x\cd
-\varphi\biggl(\sum_{k=1}^{j-1}
\dfrac{s_k\tau_k}{1-p_k}e_k\cd\biggr)\cd\biggr\|_{\lt}^2.
\end{multline*}
If $x\cd\in B_{s_1,\ldots,s_{j-1}}$, then
$$x\cd=\sum_{k=1}^{j-1}s_k\tau_ke_k\cd+z(x)\cd,\quad z(x)\cd=\sum_{k=j}^{2N}s_k(x)\tau_ke_k\cd.$$
Moreover, with every element
$$\sum_{k=1}^{j-1}s_k\tau_ke_k\cd+z(x)\cd\in B_{s_1,\ldots,s_{j-1}}$$
the set $B_{s_1,\ldots,s_{j-1}}$ contains the element
$$\sum_{k=1}^{j-1}s_k\tau_ke_k\cd-z(x)\cd.$$
Thus,
\begin{multline*}
\frac{p_j-p_{j-1}}{2^{2N}}\sum_{s_1,\ldots,s_{j-1}}\,\,\sum_{x\in B_{s_1,\ldots,s_{j-1}}}\biggl\|\Lambda x\cd-\varphi\biggl(\sum_{k=1}^{j-1}
\dfrac{s_k\tau_k}{1-p_k}e_k\cd\biggr)\cd\biggr\|_{\lt}^2\\
=\frac{p_j-p_{j-1}}{2^{2N}}\sum_{s_1,\ldots,s_{j-1}}\,\,\sum_{x\in B_{s_1,\ldots,s_{j-1}}}\biggl\|\Lambda\biggl(\sum_{k=1}^{j-1}s_k\tau_ke_k\cd+z(x)\cd\biggr)\\
-\varphi\biggl(\sum_{k=1}^{j-1}
\dfrac{s_k\tau_k}{1-p_k}e_k\cd\biggr)\cd\biggr\|_{\lt}^2\\
=\frac{p_j-p_{j-1}}{2^{2N}}\sum_{s_1,\ldots,s_{j-1}}\,\,\sum_{x\in B_{s_1,\ldots,s_{j-1}}}\biggl\|\Lambda\biggl(\sum_{k=1}^{j-1}s_k\tau_ke_k\cd\biggr)+\Lambda z(x)\cd\\
-\varphi\biggl(\sum_{k=1}^{j-1}
\dfrac{s_k\tau_k}{1-p_k}e_k\cd\biggr)\cd\biggr\|_{\lt}^2\\
=\frac{p_j-p_{j-1}}{2^{2N+1}}\sum_{s_1,\ldots,s_{j-1}}\,\,\sum_{x\in B_{s_1,\ldots,s_{j-1}}}\biggl(\biggl\|\Lambda\biggl(\sum_{k=1}^{j-1}s_k\tau_ke_k\cd\biggr)
+\Lambda z(x)\cd\\
-\varphi\biggl(\sum_{k=1}^{j-1}
\dfrac{s_k\tau_k}{1-p_k}e_k\cd\biggr)\cd\biggr\|_{\lt}^2+
\biggl\|\Lambda\biggl(\sum_{k=1}^{j-1}s_k\tau_ke_k\cd\biggr)-\Lambda z(x)\cd\\
-\varphi\biggl(\sum_{k=1}^{j-1}
\dfrac{s_k\tau_k}{1-p_k}e_k\cd\biggr)\cd\biggr\|_{\lt}^2\biggr)\\
\ge\frac{p_j-p_{j-1}}{2^{2N}}\sum_{s_1,\ldots,s_{j-1}}\,\,\sum_{x\in B_{s_1,\ldots,s_{j-1}}}\|\Lambda z(x)\cd\|_{\lt}^2\\
=\frac{p_j-p_{j-1}}{2^{2N}}\sum_{x\in B}\|\Lambda z(x)\cd\|_{\lt}^2
=(p_j-p_{j-1})\sum_{k=j}^{2N}\mu_k\tau_k^2,
\end{multline*}
where
$$\mu_{2j-1}=\int_{x_{j-1}}^{x_j}|\mu(t)|^2\,dt,\quad\mu_{2j}=
\int_{x_{-j}}^{x_{-j+1}}|\mu(t)|^2\,dt,\quad j=1,\ldots,N.$$
Substituting this estimate into \eqref{ver}, we get
\begin{multline*}
e^2(\Lambda,W,I,\delta,\varphi)\ge\sum_{j=1}^{2N+1}(p_j-p_{j-1})\sum_{k=j}^{2N}\mu_k\tau_k^2\\=
\sum_{j=1}^{2N}\biggl(p_j\sum_{k=j}^{2N}\mu_k\tau_k^2-
p_j\sum_{k=j+1}^{2N}\mu_k\tau_k^2\biggr)
=\sum_{j=1}^{2N}p_j\mu_j\tau_j^2=\sum_{j=1}^{2N}\frac{\delta^2}{\delta^2+\tau_j^2}
\mu_j\tau_j^2.
\end{multline*}
Since the method $\varphi$ was chosen arbitrarily, we have
\begin{equation}\label{Elo}
E^2(\Lambda,W,I,\delta)\ge\sup_{\substack{\tau_1\ge\ldots\ge\tau_{2N}>0\\x_\tau\cd\in W}}
\sum_{j=1}^{2N}\frac{\delta^2}{\delta^2+\tau_j^2}\mu_j\tau_j^2.
\end{equation}

The condition $x_\tau\cd\in W$ means that
$$\iR\nu(t)|x_\tau(t)|^2\,dt=\sum_{j=1}^{2N}\nu_j\tau_j^2\le1,$$
where
$$\nu_{2j-1}=\int_{x_{j-1}}^{x_j}\nu(t)\,dt,\quad\nu_{2j}=
\int_{x_{-j}}^{x_{-j+1}}\nu(t)\,dt,\quad j=1,\ldots,N.$$
Hence,
$$E^2(\Lambda,W,I,\delta)\ge\sup_{\substack{\tau_1\ge\ldots\ge\tau_{2N}>0
\\\sum_{j=1}^{2N}\nu_j\tau_j^2\le1}}
\sum_{j=1}^{2N}\frac{\delta^2}{\delta^2+\tau_j^2}\mu_j\tau_j^2.$$

Let $\tau=(\tau_1,\ldots,\tau_k,0,\ldots,0)$, $1\le k<2N$, $\tau_1\ge\ldots\ge\tau_k>0$, and
$$\sum_{j=1}^k\nu_j\tau_j^2\le1.$$
For sufficiently small $\varepsilon>0$ we put
$\tau_\varepsilon=(\tau_1(\varepsilon),\ldots,\tau_{2N}(\varepsilon))$ where
$$\tau_j(\varepsilon)=\begin{cases}\sqrt{\tau_j^2-\varepsilon},&1\le j\le k,\\
C\sqrt{\varepsilon},&k+1\le j\le2N,\end{cases}$$
and
$$C=\left(\frac{\sum_{j=1}^k\nu_j}{\sum_{j=k+1}^{2N}\nu_j}\right)^{1/2}.$$
Then
$$\sum_{j=1}^{2N}\nu_j\tau_j^2(\varepsilon)=\sum_{j=1}^k\nu_j\tau_j^2-\varepsilon\sum_{j=1}^k\nu_j+
C^2\varepsilon\sum_{j=k+1}^{2N}\nu_j=\sum_{j=1}^k\nu_j\tau_j^2\le1.$$
For $\varepsilon<\tau_k^2/(1+C^2)$ we have
$$\sqrt{\tau_k^2-\varepsilon}>C\sqrt\varepsilon.$$
Consequently, for such $\varepsilon$
$$\tau_1(\varepsilon)\ge\ldots\ge\tau_{2N}(\varepsilon)>0.$$
It follows from \eqref{Elo} that
$$E^2(\Lambda,W,I,\delta)\ge\sum_{j=1}^{2N}\frac{\delta^2}{\delta^2+\tau_j^2(\varepsilon)}
\mu_j\tau_j^2(\varepsilon).$$
Passing to the limit as $\varepsilon\to0$, we obtain
$$E^2(\Lambda,W,I,\delta)\ge\sum_{j=1}^k\frac{\delta^2}{\delta^2+\tau_j^2}
\mu_j\tau_j^2.$$
Thus,
\begin{equation}\label{Elo1}
E^2(\Lambda,W,I,\delta)\ge\sup_{\substack{\tau_1\ge\ldots\ge\tau_{2N}\ge0
\\\sum_{j=1}^{2N}\nu_j\tau_j^2\le1}}
\sum_{j=1}^{2N}\frac{\delta^2}{\delta^2+\tau_j^2}\mu_j\tau_j^2.
\end{equation}

Let the piecewise continuous function $x\cd$ be such that $|x\cd|$ is an even function monotonically decreasing on $\mathbb R_+$ and
\begin{equation}\label{nu}
\iR\nu(t)|x(t)|^2\,dt<1.
\end{equation}
Consider the integral
$$I=\iR\frac{\delta^2}{\delta^2+|x(t)|^2}|\mu(t)|^2|x(t)|^2\,dt.$$
Let us fix $\varepsilon>0$ and find $A>0$ such that
\begin{equation}\label{I1}
I_1=\int_{-A}^A\frac{\delta^2}{\delta^2+|x(t)|^2}|\mu(t)|^2|x(t)|^2\,dt>I-\varepsilon.
\end{equation}
It is obvious that
$$I_2=\int_{-A}^A\nu(t)|x(t)|^2\,dt<1.$$
We will approximate these integrals by integral sums over partitions of $T_N$, representing segments $[x_{j-1},x_j]$, $[x_{-j},x_{-j+1}]$, $j=1,\ldots,N$, and points $t_{2j-1}=\dfrac{2j-1}{2N}$, $t_{2j}=-\dfrac{2j-1}{2N}$, $j=1,\ldots,2N$.

For any $\varepsilon_1>0$, there is such a $N_1$ that for all $N>N_1$
\begin{equation}\label{frI}
\frac1N\sum_{j=1}^{2N}\frac{\delta^2}{\delta^2+\tau_j^2}|\mu(t_j)|^2\tau_j^2
>I_1-\varepsilon_1,
\end{equation}
where $\tau_j=|x(t_j)|$. Moreover, there is such a $N_2$ that for all $N>N_2$ for some $\omega>0$ the inequality
$$\frac1N\sum_{j=1}^{2N}\nu(t_j)\tau_j^2<1-\omega$$
holds.

By the mean value theorem for integrals, there are $\xi_j$ such that $\xi_{2j-1}\in[x_{j-1},x_j]$, $\xi_{2j}\in[x_{-j},x_{-j+1}]$, $j=1,\ldots,N$, and $\mu_j=|\mu(\xi_j)|^2/N$. Using the same arguments we obtain that there are $\eta_j$ such that $\eta_{2j-1}\in[x_{j-1},x_j]$, $\eta_{2j}\in[x_{-j},x_{-j+1}]$, $j=1,\ldots,N$, and $\nu_j=\nu(\eta_j)/N$. Due to the uniform continuity of the functions $|\mu\cd|$ and $\nu\cd$ on the segment $[-A,A]$ for any $\varepsilon_2>0$, there is $N_3$ such that for all $s_1,s_2\in[-A,A]$, $|s_1-s_2|<1/N_3$, inequalities
\begin{equation}\label{mn}
||\mu(s_1)|^2-|\mu(s_2)|^2|<\varepsilon_2,\quad|\nu(s_1)-\nu(s_2)|<\varepsilon_2
\end{equation}
hold.

Put
$$M=\vraisup_{t\in[0,1]}|x(t)|^2$$
(due to the monotonous decrease of $|x\cd|$ on $\mathbb R_+$, instead of the segment $[0,1]$, we may take any segment $[0,b]$, $b>0$). Choose $\varepsilon_2<\omega/(2M)$. Let $N>\max\{N_1,N_2,N_3\}$.
Then $\nu_j=\nu(\eta_j)/N<\nu(t_j)/N+\varepsilon_2/N$. Consequently,
\begin{equation}\label{ed}
\sum_{j=1}^{2N}\nu_j\tau_j^2\le\frac1N\sum_{j=1}^{2N}(\nu(t_j)+\varepsilon_2)\tau_j^2
\le\frac1N\sum_{j=1}^{2N}\nu(t_j)\tau_j^2+2M\varepsilon_2<1.
\end{equation}
It follows from \eqref{mn} that $\mu_j=|\mu(\xi_j)|^2/N>|\mu(t_j)|^2/N
-\varepsilon_2/N$. Therefore,
\begin{multline*}
\sum_{j=1}^{2N}\frac{\delta^2}{\delta^2+\tau_j^2}\mu_j\tau_j^2\ge
\frac1N\sum_{j=1}^{2N}\frac{\delta^2}{\delta^2+\tau_j^2}(|\mu(t_j)|^2
-\varepsilon_2)\tau_j^2\\
\ge\frac1N\sum_{j=1}^{2N}\frac{\delta^2}{\delta^2+\tau_j^2}|\mu(t_j)|^2
\tau_j^2-2M\varepsilon_2.
\end{multline*}
Taking into account \eqref{ed}, \eqref{frI}, and \eqref{I1}, it follows from \eqref{Elo} that
\begin{multline*}
E^2(\Lambda,W,I,\delta)\ge\sum_{j=1}^{2N}\frac{\delta^2}{\delta^2+\tau_j^2}\mu_j\tau_j^2\ge
\frac1N\sum_{j=1}^{2N}\frac{\delta^2}{\delta^2+\tau_j^2}|\mu(t_j)|^2
\tau_j^2-2M\varepsilon_2\\
\ge I_1-\varepsilon_1-2M\varepsilon_2\ge I-\varepsilon-\varepsilon_1-2M\varepsilon_2.
\end{multline*}
Due to the fact that $\varepsilon$, $\varepsilon_1$ and $\varepsilon_2$ can be chosen arbitrarily small, we get
\begin{equation}\label{Em}
E^2(\Lambda,W,I,\delta)\ge\sup_{\substack{x\cd\in W_0\\\iR\nu(t)|x(t)|^2\,dt<1}}\iR\frac{\delta^2}{\delta^2+|x(t)|^2}|\mu(t)|^2|x(t)|^2\,dt,
\end{equation}
where $W_0$ is the set of piecewise continuous functions $x\cd$ such that $|x\cd|$ is an even function monotonically decreasing on $\mathbb R_+$.

We show that the strict inequality on the right side of \eqref{Em} can be replaced by a non-strict one. Let $x\cd\in W_0$ and
$$\iR\nu(t)|x(t)|^2\,dt=1.$$
Consider the function $y\cd=(1+\varepsilon)^{-1/2}x\cd$, $\varepsilon>0$. Then it follows from \eqref{Em} that
$$E^2(\Lambda,W,I,\delta)\ge\frac1{1+\varepsilon}\iR\frac{\delta^2}{\delta^2+\dfrac{|x(t)|^2}
{1+\varepsilon}}|\mu(t)|^2|x(t)|^2\,dt.$$
Since
$$\frac1{\delta^2+\dfrac{|x(t)|^2}{1+\varepsilon}}\ge\frac1{\delta^2+|x(t)|^2},$$
we have
$$E^2(\Lambda,W,I,\delta)\ge\frac1{1+\varepsilon}\iR\frac{\delta^2}{\delta^2+|x(t)|^2}
|\mu(t)|^2|x(t)|^2\,dt.$$
Passing $\varepsilon$ to zero, we get
$$E^2(\Lambda,W,I,\delta)\ge\iR\frac{\delta^2}{\delta^2+|x(t)|^2}|\mu(t)|^2|x(t)|^2\,dt.$$
Thus,
\begin{equation}\label{El}
E^2(\Lambda,W,I,\delta)\ge\sup_{\substack{x\cd\in W_0\\\iR\nu(t)|x(t)|^2\,dt\le1}}\iR\frac{\delta^2}{\delta^2+|x(t)|^2}|\mu(t)|^2|x(t)|^2\,dt.
\end{equation}

2. The upper bound. Let us find the error of the methods having the form
$$\varphi(y_\xi\cd)\cd=\alpha\cd\mu\cd y_\xi\cd.$$
Put $z_\xi\cd=y_\xi\cd-x\cd$. Then $\mathbb Mz_\xi\cd=0$, $\mathbb Dz_\xi\cd\le\delta^2$. We have
\begin{multline*}
e^2(\Lambda,W,I,\delta,\varphi)=\sup_{\substack{x\cd\in W\\y_\xi\cd\in Y_\delta(x\cd)}}\mathbb M\left(\|\Lambda x\cd-\varphi(y_\xi\cd)\cd\|^2_{\lt}\right)\\
=\sup_{\substack{x\cd\in W\\y_\xi\cd\in Y_\delta(x\cd)}}\mathbb M\left(\|\Lambda x\cd-\varphi(x\cd)\cd-\varphi(z_\xi\cd)\cd\|^2_{\lt}\right)\\
=\sup_{\substack{x\cd\in W\\y_\xi\cd\in Y_\delta(x\cd)}}\left(\|\Lambda x\cd-\varphi(x\cd)\cd\|^2_{\lt}+\mathbb M(\|\varphi(z_\xi\cd)\cd\|^2_{\lt})\right.\\
\left.-2\mathbb M\RE(\varphi(z_\xi\cd)\cd,\ov{\Lambda x\cd-\varphi(x\cd)\cd})\right);
\end{multline*}
here $(\cdot,\cdot)$ is the standard scalar product in $\lt$. It follows from the form of $\varphi$ that
\begin{multline*}
\mathbb M\RE(\varphi(z_\xi\cd)\cd,\Lambda x\cd-\varphi(x\cd)\cd)\\
=\RE\mathbb M\left(\alpha\cd\mu\cd z_\xi\cd,\ov{\Lambda x\cd-\varphi(x\cd)\cd}\right)\\
=\RE\left(\alpha\cd\mu\cd\mathbb Mz_\xi\cd,\ov{\Lambda x\cd-\varphi(x\cd)\cd}\right)=0.
\end{multline*}
Due to the fact that $\mathbb M(|z_\xi\cd|^2)=\mathbb D(y_\xi\cd)$, we have
\begin{multline*}
e^2(\Lambda,W,I,\delta,\varphi)=\sup_{\substack{x\cd\in W\\y_\xi\cd\in Y_\delta(x\cd)}}
\left(\iR|\mu(t)|^2|1-\alpha(t)|^2|x(t)|^2\,dt\right.\\
\left.+\iR|\mu(t)|^2|\alpha(t)|^2\mathbb D(y_\xi(t))\,dt\right)=
\sup_{x\cd\in W}\iR|\mu(t)|^2|1-\alpha(t)|^2|x(t)|^2\,dt\\
+\delta^2\iR|\mu(t)|^2|\alpha(t)|^2\,dt
\end{multline*}
Since
\begin{multline*}
\iR|\mu(t)|^2|1-\alpha(t)|^2|x(t)|^2\,dt=\iR\frac{|\mu(t)|^2}{\nu(t)}|1-\alpha(t)|^2
\nu(t)|x(t)|^2\,dt\\
\le\vraisup_{t\in\mathbb R}\left(\frac{|\mu(t)|^2}{\nu(t)}|1-\alpha(t)|^2\right),
\end{multline*}
we obtain
$$e^2(\Lambda,W,I,\delta,\varphi)\le\vraisup_{t\in\mathbb R}\left(\frac{|\mu(t)|^2}{\nu(t)}|1-\alpha(t)|^2\right)+\delta^2\iR|\mu(t)|^2|\alpha(t)|^2\,dt.$$

Put
$$\alpha(t)=\left(1-\dfrac{\sqrt{\nu(t)}}{|\mu(t)|}\frac{|\mu(t_\delta)|}{\sqrt{\nu(t_\delta)}}
\right)_+.$$
Due to the monotonous decreasing of the function $|\mu\cd|/\sqrt{\nu\cd}$ we get
$$\vraisup_{t\in\mathbb R}\left(\frac{|\mu(t)|^2}{\nu(t)}|1-\alpha(t)|^2\right)
=\frac{|\mu(t_\delta)|^2}{\nu(t_\delta)}.$$
Consequently,
\begin{multline*}
e^2(\Lambda,W,I,\delta,\varphi)\le\frac{|\mu(t_\delta)|^2}{\nu(t_\delta)}\\+\delta^2\int_{|t|\le t_\delta}|\mu(t)|^2\biggl(1-\dfrac{\sqrt{\nu(t)}}{|\mu(t)|}\frac{|\mu(t_\delta)|}
{\sqrt{\nu(t_\delta)}}\biggr)^2\,dt
=\frac{|\mu(t_\delta)|^2}{\nu(t_\delta)}\\+\delta^2\int_{|t|\le t_\delta}|\mu(t)|^2\biggl(\biggl(1-\dfrac{\sqrt{\nu(t)}}{|\mu(t)|}\frac{|\mu(t_\delta)|}
{\sqrt{\nu(t_\delta)}}\biggr)-\dfrac{\sqrt{\nu(t)}}{|\mu(t)|}\frac{|\mu(t_\delta)|}
{\sqrt{\nu(t_\delta)}}\\
+\dfrac{\nu(t)}{|\mu(t)|^2}\frac{|\mu(t_\delta)|^2}
{\nu(t_\delta)}\biggr)\,dt=\delta^2\int_{|t|\le t_\delta}|\mu(t)|^2\biggl(1-\dfrac{\sqrt{\nu(t)}}{|\mu(t)|}\frac{|\mu(t_\delta)|}
{\sqrt{\nu(t_\delta)}}\biggr)\,dt\\
+\frac{|\mu(t_\delta)|}{\sqrt{\nu(t_\delta)}}\biggl(\frac{|\mu(t_\delta)|}{\sqrt{\nu(t_\delta)}}
\biggl(1+\delta^2\int_{|t|\le t_\delta}\nu(t)\,dt\biggr)-\delta^2\int_{|t|\le t_\delta}|\mu(t)|\sqrt{\nu(t)}\,dt\biggr).
\end{multline*}
It follows from the definition of $t_\delta$ that
$$\frac{|\mu(t_\delta)|}{\sqrt{\nu(t_\delta)}}
\biggl(1+\delta^2\int_{|t|\le t_\delta}\nu(t)\,dt\biggr)-\delta^2\int_{|t|\le t_\delta}|\mu(t)|\sqrt{\nu(t)}\,dt=0.$$
Thus,
\begin{equation}\label{eb}
e^2(\Lambda,W,I,\delta,\varphi)\le\delta^2\int_{|t|\le t_\delta}|\mu(t)|^2\biggl(1-\dfrac{\sqrt{\nu(t)}}{|\mu(t)|}\frac{|\mu(t_\delta)|}
{\sqrt{\nu(t_\delta)}}\biggr)\,dt.
\end{equation}

Consider the function
$$\wx(t)=\delta\biggl(\biggl(\frac{\sqrt{\nu(t_\delta)}}{|\mu(t_\delta)|}\frac{|\mu(t)|}
{\sqrt{\nu(t)}}-1\biggr)_+\biggr)^{1/2}.$$
It is obvious that $\wx\cd\in W_0$. Moreover,
$$\iR\nu(t)|\wx(t)|^2\,dt=\delta^2\int_{|t|\le t_\delta}\nu(t)\biggl(\frac{\sqrt{\nu(t_\delta)}}{|\mu(t_\delta)|}\frac{|\mu(t)|}
{\sqrt{\nu(t)}}-1\biggr)\,dt=1.$$
Taking into account \eqref{eb}, it follows from \eqref{El} that
\begin{multline*}
E^2(\Lambda,W,I,\delta)\ge\iR\frac{\delta^2}{\delta^2+|\wx(t)|^2}|\mu(t)|^2|\wx(t)|^2\,dt\\
=\delta^2\int_{|t|\le t_\delta}|\mu(t)|^2\biggl(1-\dfrac{\sqrt{\nu(t)}}{|\mu(t)|}\frac{|\mu(t_\delta)|}
{\sqrt{\nu(t_\delta)}}\biggr)\,dt\\
\ge e^2(\Lambda,W,I,\delta,\varphi)\ge E^2(\Lambda,W,I,\delta).
\end{multline*}
This implies \eqref{TE} and the optimality of the method $\varphi$.
\end{proof}

Now we consider some examples of the application of Theorem~\ref{T1}.

\section[Recovery of functions and their derivatives]{Recovery of functions and their derivatives from the Fourier transform given with random error}

Denote by $\mathcal W_2^r(\mathbb R)$ the set of functions $x\cd\in\lt$ for which
$x^{(r-1)}\cd$ is locally absolutely continuous and $x^{(r)}\cd\in\lt$. Put
$$\Wr=\{\,x\cd\in\mathcal W_2^r(\mathbb R):\|x^{(r)}\cd\|_{\lt}\le1\,\}.$$
Suppose that the Fourier transform $Fx\cd$ of the function $x\cd\in\Wr$ is given with a random error. We assume that instead of the function $Fx\cd$ we know a random function $y_\xi\cd\in\lt$ such that $\mathbb My_\xi\cd=Fx\cd$ and $\mathbb Dy_\xi\cd\le\delta$. Using this information, it is required to recover the function $x^{(k)}\cd$, $0\le k<r$, in the $\lt$-metric.

The exact setting of the problem is as follows. For every $x\cd\in\Wr$ we consider the set of random functions
\begin{equation}\label{Yd}
Y_\delta(x\cd)=\{\,y_\xi\cd\in\lt:\mathbb My_\xi\cd=Fx\cd,\ \mathbb Dy_\xi\cd\le\delta\,\}.
\end{equation}
This set requires additional conditions given in the general setting (the validity of the equalities \eqref{eq1} and \eqref{eq2}). Next, we define the error of the recovery method $\varphi\colon\lt\to\lt$ as follows
$$e(D^k,\Wr,\delta,\varphi)
=\left(\sup_{\substack{x\cd\in\Wr\\y_\xi\cd\in Y_\delta(x\cd)}}\mathbb M\left(\|x^{(k)}\cd-\varphi(y_\xi\cd)\cd\|^2_{\lt}\right)\right)^{1/2}.$$
The problem is to find the error of optimal recovery
$$E(D^k,\Wr,\delta)=\inf_{\varphi\colon\lt\to\lt}e(D^k,\Wr,\delta,\varphi)$$
and a method on which this infimum is attained.

It follows from the Parseval equality that
\begin{gather*}
\|x^{(r)}\cd\|_{\lt}^2=\frac1{2\pi}\iR|t|^{2r}|Fx(t)|^2\,dt,\\
\|x^{(k)}\cd-\varphi(y_\xi\cd)\cd\|^2_{\lt}=\frac1{2\pi}\iR|(it)^kF(t)-
F\varphi(y_\xi\cd)(t)|^2\,dt.
\end{gather*}
Thus, the problem is reduced to problem \eqref{bas} with $\nu(t)=t^{2r}$ and $\mu(t)=(it)^k$.
The function $f\cd$ which was defined by \eqref{fs}, has the form
$$f(s)=\int_{|t|\le s}\left(\frac{s^{r-k}}{|t|^{r-k}}-
1\right)t^{2r}\,dt=2s^{2r+1}\frac{r-k}{(2r+1)(r+k+1)}.$$
The equation $f(s)=\delta^{-2}$ has the unique solution
$$t_\delta=\left(\frac{(2r+1)(r+k+1)}{2\delta^2(r-k)}\right)^{\frac1{2r+1}}.$$
It follows from Theorem~\ref{T1}

\begin{theorem}
For all $\delta>0$ and $0\le k<r$
$$E(D^k,\Wr,\delta)=\frac{(2r+1)^{\frac{2k+1}{2(2r+1)}}}{\sqrt{2k+1}}\left(2\delta^2
\frac{r-k}{r+k+1}\right)^{\frac{r-k}{2(2r+1)}}.$$
Moreover, the method
$$\varphi(y_\xi\cd)(t)=F^{-1}\left((it)^k\alpha(t)y_\xi(t)\right)(t),$$
where
$$\alpha(t)=\left(1-t^{r-k}\left(\frac{2\delta^2(r-k)}{(2r+1)(r+k+1)}\right)^{\frac{r-k}{2r+1}}
\right)_+,$$
is optimal.
\end{theorem}

Note that the optimal recovery method does not use all the information about random functions $y_\xi\cd$, but only the information contained in the segment $[-t_\delta,t_\delta]$. Moreover, the more accurate the measurements (the smaller the variance $\delta^2$), the larger this segment becomes.

The deterministic case of this problem was considered in \cite{Most} (see also \cite{Most1}). 

\section{Recovery of the solution of the heat equation}

The temperature distribution in an infinite rod is described by the equation
$$\frac{\partial u}{\partial t}=\frac{\partial^2u}{\partial t^2},$$
where $u(\cdot,\cdot)$ is the function on $[0,\infty)\times\mathbb R$) with
a given initial temperature distribution
$$u(0,\cdot)=u_0\cd.$$

Consider the problem of recovering the temperature distribution at an instant of time $T$ from information about the Fourier transform of the initial temperature distribution $u_0\cd$, given with a random error. We assume that the functions $u_0\cd$, given the initial temperature distribution, belong to the class $\Wr$. We are again considering a set of random functions similar to the set of $\eqref{Yd}$
$$Y_\delta(u_0\cd)=\{\,y_\xi\cd\in\lt:\mathbb My_\xi\cd=Fu_0\cd,\ \mathbb Dy_\xi\cd\le\delta\,\}$$
with the same additional conditions regarding the change in the order of integration.
We define the error of the recovery method $\varphi\colon\lt\to\lt$ as follows
$$e(T,\Wr,\delta,\varphi)
=\left(\sup_{\substack{u_0\cd\in\Wr\\y_\xi\cd\in Y_\delta(u_0\cd)}}\mathbb M\left(\|u(T,\cdot)-\varphi(y_\xi\cd)\cd\|^2_{\lt}\right)\right)^{1/2}.$$
The problem is to find the error of optimal recovery
$$E(T,\Wr,\delta)=\inf_{\varphi\colon\lt\to\lt}e(T,\Wr,\delta,\varphi)$$
and a method on which this infimum is attained.

It is well known (see, for example, \cite{KF}) that for all $t\ge0$ the equality
$$F(u(t,\cdot))(\lambda)=e^{-\lambda^2t}Fu_0(\lambda)$$
holds. It follows from the Parseval equality that
$$\|u(T,\cdot)-\varphi(y_\xi\cd)\cd\|^2_{\lt}=\frac1{2\pi}\iR|e^{-\lambda^2T}Fu_0(\lambda)-
F\varphi(y_\xi\cd)(\lambda)|^2\,d\lambda.$$
Thus, the problem is reduced to problem \eqref{bas} with $\nu(t)=t^{2r}$ and $\mu(t)=e^{-t^2T}$.
The function $f\cd$ which was defined by \eqref{fs}, has the form
$$f(s)=\int_{|t|\le s}\left(\frac{s^re^{s^2T}}{t^re^{t^2T}}-
1\right)t^{2r}\,dt=2s^re^{s^2T}\int_0^st^re^{-t^2T}\,dt-\frac{s^{2r+1}}{2r+1}.$$
It is easy to verify that $f(s)\to+\infty$ as $s\to+\infty$ (the monotonous increase of $f\cd$ was noted in the general case). Therefore, the equation $f(s)=\delta^{-2}$ has a unique solution, which we denote by $t_\delta$.

From Theorem~\ref{T1} we obtain the following result:

\begin{theorem}
Put
$$\alpha(t)=\left(1-\frac{t^re^{t^2T}}{t_\delta^re^{t_\delta^2T}}\right)_+.$$
Then the equality
$$E(T,\Wr,,\delta)=\delta\biggl(\int_{|t|\le t_\delta}e^{-2t^2T}\alpha(t)\,dt\biggr)^{1/2}$$
holds. Moreover, the method
$$\varphi(y_\xi\cd)(t)=F^{-1}\left(e^{-t^2T}\alpha(t)y_\xi(t)\right)(t)$$
is optimal.
\end{theorem}

The deterministic case of this problem was considered in \cite{Most2}.

\end{document}